\documentclass[12pt]{amsart}
\usepackage{amssymb}
\usepackage{amscd}
\usepackage{amsmath}
\usepackage{amsthm}

\let\geq\geqslant
\let\leq\leqslant

\DeclareMathOperator{\vol}{vol}

\DeclareMathOperator{\supp}{supp}

\DeclareMathOperator{\Ran}{Ran}
\DeclareMathOperator{\Dom}{Dom}

\DeclareMathOperator{\Tr}{Tr}

\DeclareMathOperator{\Span}{span}

\DeclareMathOperator{\Clos}{Clos}
\DeclareMathOperator{\Spec}{Spec}

\renewcommand\Im{\hbox{{\rm Im}}\,}

\newcommand{\abs}[1]{\lvert#1\rvert}
\newcommand{\aabs}[1]{\left\lvert#1\right\rvert}
\newcommand{\norm}[1]{\lVert#1\rVert}

\newcommand{\R}{{\mathbb R}}

\newcommand{\C}{{\mathbb C}}

\numberwithin{equation}{section}
\renewcommand{\theequation}{\thesection.\arabic{equation}}

\theoremstyle{plain}
\newtheorem{theorem}{\bf Theorem}[section]
\newtheorem{lemma}[theorem]{\bf Lemma}
\newtheorem{proposition}[theorem]{\bf Proposition}

\theoremstyle{definition}

\theoremstyle{remark}
\newtheorem*{remark*}{\bf Remark}
\newtheorem{remark}[theorem]{\bf Remark}
\newtheorem{example}[theorem]{\bf Example}
\newtheorem{assumption}[theorem]{\bf Assumption}

\DeclareMathOperator{\Symp}{Symp}
\DeclareMathOperator{\Ham}{Ham}
\DeclareMathOperator{\CAL}{CAL}
\newcommand{\wt}{\widetilde}
\newcommand{\N}{{\mathcal N}}
\newcommand{\M}{{\mathcal M}}

\begin{document}
\title[Scattering matrix in Hamiltonian mechanics]{The scattering matrix and associated formulas in Hamiltonian mechanics}
\sloppy

\author{Vladimir Buslaev}
\address{Dept. of Mathematical Physics\\
Institute for Physics, Saint Petersburg State University\\
1 Ulyanovskaya str, St. Petersburg-Petrodvorets\\
198504, Russia}
\curraddr{}
\email{buslaev@mph.phys.spbu.ru}
\thanks{}

\author{Alexander Pushnitski}
\address{Department of Mathematics\\
King's College London\\ 
Strand, London WC2R  2LS, U.K.}
\curraddr{}
\email{alexander.pushnitski@kcl.ac.uk}
\thanks{}



\begin{abstract}
We survey the basic notions of scattering theory in Hamiltonian mechanics
with a particular attention to the analogies with scattering theory in quantum mechanics. 
We discuss the scattering symplectomorphism, which is analogous to the scattering matrix.
We prove identities which relate the Calabi invariant of the scattering symplectomorphism
to the total time delay and the regularised phase space volume. 
These identities are analogous to the Birman-Krein formula and the Eisenbud-Wigner 
formula in quantum scattering theory. 
\end{abstract}

\maketitle

\section{Introduction}\label{sec.1}

The scattering theory in quantum mechanics deals with the following abstract framework 
(see e.g. \cite{RS3,Yafaev}): for self-adjoint operators $H_0$ and $H$ in a Hilbert space, 
the large $t$ asymptotics of the corresponding unitary groups $e^{-itH_0}$ and $e^{-itH}$ are compared. 
In the scattering theory in Hamiltonian mechanics (in its most general form), one considers two Hamiltonian
functions $H_0$ and $H$ on a non-compact symplectic manifold and compares the large time 
asymptotics of the corresponding two Hamiltonian flows. 
We refer to these two branches of scattering theory as ``quantum" and ``classical" cases for short.

One of the fundamental objects in the ``quantum" scattering theory is the \emph{scattering matrix}. 
The purpose of this paper is to discuss an object in Hamiltonian scattering theory which is in many ways
(perhaps not yet entirely understood) analogous to the scattering matrix. 
For the want of a better term, we call this object the \emph{scattering symplectomorphism}; 
it is a symplectic diffeomorphism on the manifold of 
the orbits of the Hamiltonian flow of  $H_0$ of constant energy.

In ``quantum" scattering theory, the determinant of the scattering matrix is related 
to the spectral shift function by the Birman-Krein 
formula and to the total time delay by the Eisenbud-Wigner formula. 
We will discuss the ``classical" analogues
of time delay and the spectral shift function and of the Birman-Krein and Eisenbud-Wigner formulas. 
It appears that the ``classical'' analogue of the determinant of the scattering matrix
is given by the Calabi invariant of the scattering symplectomorphism.

Some of the constructions presented here are new (to the best of our knowledge) while others appeared
in mathematics or physics literature at different levels of rigour and generality;
this will be discussed below in more detail. 
We hope that collecting all this material and presenting it in a uniform way in a fairly general 
setting will be useful. 

In Section~\ref{sec.2} we describe the set-up of Hamiltonian scattering and introduce the main
objects: the ``classical" analogues of the wave operators and the scattering map, the scattering symplectomorphism, 
the total time delay and the regularised phase space volume (the latter is the analogue of the
``quantum" spectral shift function). 
In Section~\ref{sec.d}, we state our main results which relate the scattering symplectomorphism 
to the regularised phase space volume and the time delay;  these are the ``classical" analogues of the 
Birman-Krein and the Eisenbud-Wigner formulas. 
At the end of Section~\ref{sec.d}, we also address the issue of whether 
the scattering symplectomorphism is a Hamiltonian symplectomorphism. 
In Section~\ref{sec.k}, we consider an example from classical mechanics. 
The proofs are presented in Sections~\ref{sec.g}--\ref{sec.i}. 
In Appendix~A we collect the relevant formulas and definitions from ``quantum'' scattering
theory for the purposes of comparison with the ``classical'' case.  
In Appendix~B, we recall the necessary background information from symplectic geometry. 
In particular, we recall the definition of the Calabi invariant of a symplectomorphism. 

In the authors' personal view of scattering theory, the analogy 
between the ``classical" and ``quantum" cases plays an important role. 
However, the reader not interested in the ``quantum" scattering theory, can safely 
ignore all references to it.

\section{The main objects of  scattering theory}\label{sec.2}

\subsection{Notation and assumptions}
Let $\N$ be a non-compact $2n$-dimensional symplectic manifold, $n\geq1$, with a symplectic
form $\omega$. 
We will compare the large time behaviour of the Hamiltonian flows associated with two Hamiltonian 
functions $H_0,H\in C^\infty(\N)$. 
We use the following notation: $X$ is the Hamiltonian vector field corresponding to $H$ and 
$\Phi_t:\N\to \N$  is the Hamiltonian flow: 
$$
(i(X)\omega)(\cdot)\equiv \omega(X,\cdot)=-dH(\cdot);
\quad
\frac{d}{dt}\Phi_t(\cdot)=X(\Phi_t(\cdot)),
\quad
\Phi_0=id.
$$
For $E\in\R$, let 
\begin{equation}
G(E)=\{x\in \N\mid H(x)\leq E\},
\quad 
A(E)=\{x\in \N\mid H(x)=E\}.
\label{a15}
\end{equation}
We will often suppress the dependence on $E$ in our notation, since the value of $E$ will be
fixed for a large part of the paper.
Notation $X_0$, $\Phi_t^0$, $G_0$, $A_0$ have the same meaning and refer to the Hamiltonian $H_0$. 

The symplectic volume $\vol(\Omega)$ of an open set $\Omega\subset \N$ 
is defined as
$$
\vol(\Omega)=\int_\Omega \frac{\omega^n}{n!}, \quad \omega^n=\omega\wedge\cdots\wedge\omega;
$$
note the normalisation $1/n!$. The characteristic function of $\Omega$ is denoted by $\chi_\Omega$. 

Fix $E\in\R$. We make the following assumptions: 
\begin{assumption}\label{ass1}

(i) $E$ is a regular value of $H$, $H_0$: 
$dH(x)\not=0$ for all $x\in A(E)$ and $dH_0(x)\not=0$ for all $x\in A_0(E)$. 
Thus, $A(E)$ and $A_0(E)$ are $C^\infty$-smooth manifolds. 

(ii) The maps $\Phi_t^0: A_0(E)\to A_0(E)$ and  $\Phi_t:A(E)\to A(E)$ are well defined for 
all $t\in\R$, i.e. the trajectories do not run off to infinity in finite time. 
Thus, $\Phi^0_t$ and $\Phi_t$ are groups of diffeomorphisms on $A_0(E)$ 
and $A(E)$. 

(iii) For any compact set $K\subset A_0(E)$  
there exists $T>0$ such that for all $x\in K$ and all $\abs{t}\geq T$, one has
$\Phi_t^0(x)\notin K$.

(iv) The sets $G(E)\cap\supp(H-H_0)$ and $G_0(E)\cap\supp(H-H_0)$ are compact. 
\end{assumption}

\subsection{Wave operators and the scattering map}
For $x\in A_0(E)$, let 
\begin{equation}
W_\pm(x)=\lim_{t\to\pm\infty}\Phi_{-t}\circ\Phi_t^0(x).
\label{wo}
\end{equation}
If $K_0\subset A_0(E)$ is a compact, then, 
taking $K=\supp(H-H_0)\cup K_0$ in Assumption~\ref{ass1}(iii), we see that  the above limits exist 
and are attained at finite values of $t$:
\begin{equation}
W_+(x)=\Phi_{-t}\circ\Phi^0_t(x), 
\quad
W_-(x)=\Phi_{t}\circ\Phi^0_{-t}(x), 
\quad 
\forall x\in K_0,
\quad
\forall \abs{t}\geq T.
\label{a7}
\end{equation}
Since $H_0$ (resp. $H$) is constant along the orbits of $\Phi^0$ (resp. $\Phi$), we get
$$
H(W_+(x))=H(\Phi_{-t}\circ\Phi^0_t(x))=H(\Phi^0_t(x))=H_0(\Phi^0_t(x))=H_0(x),
\quad 
\forall x\in K_0,
\quad
\forall \abs{t}\geq T
$$
and in the same way, $H(W_-(x))=H_0(x)$. It follows that $W_\pm(A_0(E))\subset A(E)$. 
However, it is easy to construct examples such that $W_\pm(A_0(E))\not=A(E)$. 
Thus, we make an additional assumption: 
\begin{equation}
W_+(A_0(E))=W_-(A_0(E))=A(E).
\label{completeness}
\end{equation}
Since $\Phi_t$ and $\Phi_t^0$ are symplectic 
diffeomorphisms of $\N$ for each $t$, it follows from \eqref{a7} that $W_\pm:A_0(E)\to A(E)$ 
are diffeomorphisms and that
$W_\pm^*(\omega|_{A})=\omega|_{A_0}$. 

Next, since the definition of $W_\pm$ can also be written as 
$W_\pm(x)=\lim_{s\to\pm\infty}\Phi_{-t-s}\circ\Phi_{t+s}^0(x)$,
we get the intertwining property
\begin{equation}\label{a2}
W_\pm\circ \Phi^0_t=\Phi_t\circ W_\pm, 
\quad \forall t\in\R.
\end{equation}
Assumption \eqref{completeness}, together with the intertwining property \eqref{a2}, 
ensures that the flow $\Phi_t$ on $A(E)$ has essentially the same properties as $\Phi_t^0$ on 
$A_0(E)$.
In particular, Assumption~\ref{ass1}(iii) holds true for the flow $\Phi_t$ on $A(E)$.

Assuming \eqref{completeness}, we can define the scattering map 
\begin{equation}
S_E=W_+^{-1}\circ W_-.
\label{sm}
\end{equation}
By \eqref{a7}, one can write
\begin{equation}
S_E(x)=\Phi_{-t}^0\circ \Phi_{2t}\circ \Phi_{-t}^0(x), 
\quad 
\forall x\in K_0,
\quad
\forall \abs{t}\geq T.
\label{a1}
\end{equation}
It follows that $S_E: A_0(E)\to A_0(E)$ is a diffeomorphism onto $A_0(E)$ and 
\begin{equation}
S_E^*(\omega |_{A_0})=\omega |_{A_0}. 
\label{a3}
\end{equation}
From \eqref{a2} (or directly from \eqref{a1}) it follows that 
\begin{equation}
S_E\circ \Phi^0_t=\Phi_t^0\circ S_E, 
\quad \forall t\in\R.
\label{a10}
\end{equation}
The scattering map is usually defined initially on the whole of $\N$ (or for some range of energies)
and then restricted onto $A_0(E)$. 

The above constructions are very well known; see e.g. \cite{Hunziker,RS3,Thirring2}.
The fact that the wave operators and the scattering map are symplectic transformations
is particularly  emphasized in the works by W.~Thirring, see \cite{Thirring2} or
\cite{NT,Thirring1}.

\subsection{Symplectic reduction and the scattering symplectomorphism}
One can consider the set of all orbits of the dynamics $\Phi^0$ on the constant energy surface $A_0(E)$ 
as a symplectic manifold $\wt A_0=\wt A_0(E)$. 
Indeed, by Assumption~\ref{ass1}(i)--(iii),  the action of the group $\Phi^0$ on $A_0$ is smooth, proper and free, 
and therefore
(see \cite[Proposition~4.1.23]{AbrM}) the orbit space admits a smooth manifold structure 
and the quotient map $\pi_0: A_0\to \wt A_0$ is a submersion. 
It is easy to construct charts on $\wt A_0$ by choosing sufficiently small $(2n-2)$-dimensional submanifolds
of $A_0$ such that $X_0$ is non-tangential to these manifolds; see the proof of Lemma~\ref{lma.g1}.
If $x\in A_0$ is a point of an orbit $y\in \wt A_0$, then the tangent space $T_y \wt A_0$ can be identified 
with the quotient space $T_x A_0/\Span\{ X_0(x)\}$.
There exists a unique symplectic form $\wt \omega_0$ on $\wt A_0$ such that 
$\pi_0^*\wt \omega_0=\omega|_{A_0}$; 
see e.g. \cite[Theorem~4.3.1 and Example~4.3.4(ii)]{AbrM}.
It is not difficult to prove that if $\N$ is exact (i.e. there exists a 1-form $\alpha$ on $\N$ such 
that $\omega=d\alpha$), then $\wt A_0$ is also exact, see Lemma~\ref{lma.g1} below.

If $f:A_0\to \R$ is a smooth function such that $f\circ \Phi_t^0=f$ for all $t\in\R$, then $f$ generates
a smooth function $\wt f: \wt A_0\to\R$ such that $\wt f\circ \pi_0=f$. 
In a similar way, by \eqref{a10}, the scattering map $S_E$ generates the map 
\begin{equation}
\wt S_E: \wt A_0\to \wt A_0, \quad
\pi_0\circ S_E=\wt S_E\circ \pi_0.
\label{a11}
\end{equation}
We will call $\wt S_E$ the \emph{scattering symplectomorphism}.

Since the action of $\Phi$ on $A$  is also free, smooth and proper, one can 
consider the symplectic manifold $\wt A$ of the orbits of $\Phi$ on $A$, with the natural 
projection $\pi: A\to\wt A$ and a symplectic form $\wt \omega$ on $\wt A$. By the intertwining property 
\eqref{a2}, there exist symplectic diffeomorphisms 
\begin{equation}
\wt W_\pm: \wt A_0\to\wt A, 
\quad 
\wt W_\pm\circ \pi_0=\pi\circ W_\pm. 
\label{a12}
\end{equation}

Since we are going to discuss integration of forms over $\N$, $A_0$, $A$, $\wt A_0$, $\wt A$, 
we should fix orientation on these manifolds. 
Orientation on $\N$ is fixed in such a way that the form $\omega^n$ is positive on a positively 
oriented basis. In the same way, orientation on $\wt A_0$ is fixed in such a way that the form
$\wt \omega_0^{n-1}$ is positive on a positively oriented basis. 
Orientation on $A_0$ is fixed such that if $(e_1,\dots,e_{2n-1})$ is a positively oriented
basis in $T_x A_0$ and $\xi\in T_x \N$ is such that $d_xH_0(\xi)>0$, then 
$(\xi,e_1,\dots,e_{2n-1})$ is a positively oriented basis in $T_x \N$. 
In other words, $A_0$ is considered as a boundary of $G_0$
with induced orientation.
Orientation on $\wt A$ and $A$ is fixed in a similar way to $\wt A_0$, $A_0$.

In the case $n=1$ the above reduction produces a ``manifold'' of dimension zero, 
i.e. a discrete set of orbits. The scattering symplectomorphism becomes
just a permutation map on the set of these orbits. In this case, integration over 
the ``volume forms''
$\wt \omega^0$, $\wt \omega_0^0$ will be understood simply as summation over this 
set of orbits.

\subsection{Poincar\'e section}\label{sec.a4}
The above procedure of symplectic reduction looks particularly simple if one makes 
\begin{assumption}\label{ass2}
There exists a smooth submanifold $\Gamma\subset A_0$ of dimension $2n-2$ such that:

(a) $X_0(x)\notin T_x\Gamma$ for all $x\in\Gamma$; 

(b) for all $x\in A_0$, there exists a unique $z=z(x)\in\Gamma$ 
and a unique $t=t(x)$ such that $x=\Phi_t^0(z)$. 
\end{assumption}
In this case, the elements $x\in A_0$ can be considered as pairs $(z,t)\in\Gamma\times \R$ 
such that $x=\Phi_t(z)$.
It is easy to see that 
$$
\text{Assumption~\ref{ass1}(i),(ii)}+\text{Assumption~\ref{ass2}}\Rightarrow \text{Assumption~\ref{ass1}(iii)}.
$$
 Let $i: \Gamma\to A_0$ be the natural embedding. 
Then $\gamma_0=\pi_0\circ i: \Gamma\to\wt A_0$ is a diffeomorphism and 
$\gamma_0^*\wt \omega_0=\omega|_\Gamma$. Thus, $\Gamma$ can be considered
as a ``concrete" realisation of the ``abstract" manifold $\wt A_0$. 

Using the above identification of $A_0$ and $\Gamma\times\R$, the free dynamics can be 
represented as
$\Phi_s^0:(z,t)\mapsto (z,t+s)$ and the scattering map as 
\begin{equation}
S_E:(z,t)\mapsto (\wt s_E(z), t-\tau_E(z)), 
\label{a13}
\end{equation}
where $\wt s_E=\gamma_0^{-1}\circ \wt S_E\circ \gamma_0: \Gamma\to\Gamma$ 
is a symplectic diffeomorphism, and $\tau_E: \Gamma\to\R$ is a smooth function. 
The map $\tau_E$ is often called \emph{time delay}, or \emph{sojourn time}. 
We note that the definition \eqref{a13} of $\tau_E$ depends on the choice of $\Gamma$; 
there is no invariant way of defining a time delay function on $\wt A_0$. 

The manifold $\Gamma$ is, of course, the well known Poincar\'e section; 
see e.g. \cite[\S 7.1]{AbrM}. The map $\wt s_E: \Gamma\to\Gamma$ in concrete 
cases appeared before in physics literature under the name Poincar\'e scattering map; 
see  \cite{Jung,Uzy,Uzy2}. Its connection with the ``quantum'' scattering matrix 
has also been discussed in physics literature, see e.g. \cite{Uzy2}.

\subsection{Total time delay}\label{sec.ttd}
Although the definition of time delay $\tau_E$ above depends on the choice of $\Gamma$, 
the \emph{total} (or average) \emph{time delay} $T_E$ can be defined (see \eqref{e6} below) in an invariant way.

Let $\Omega_1\subset\Omega_2\subset\dots\subset \N$ be a sequence of open 
pre-compact sets such that $\cup_{k=1}^\infty \Omega_k=\N$. 
Let us define the functions $u_k^0:A_0\to\R$ and $u_k:A\to\R$ by 
\begin{align}
u_k^0(x)&
=
\int_{-\infty}^\infty 
\chi_{\Omega_k}\circ \Phi_t^0(x) dt,
\quad x\in A_0,
\\
u_k(x)&
=\int_{-\infty}^\infty 
\chi_{\Omega_k}\circ\Phi_t(x)dt,
\quad x\in A.
\end{align}
It is straightforward to see that $u_k^0\circ \Phi_t^0=u_k^0$ and $u_k\circ \Phi_t=u_k$ 
for all $t\in\R$ and therefore $u_k^0$, $u_k$ generate functions $\wt u_k^0:\wt A_0\to\R$, 
$\wt u_k:\wt A\to\R$ such that
$\wt u_k^0\circ \pi_0=u_k^0$ and $\wt u_k\circ \pi=u_k$.

\begin{theorem}\label{th4}
Suppose Assumption~\ref{ass1} and completeness \eqref{completeness} hold true, and 
let $\Omega_1\subset\Omega_2\subset\dots\subset \N$ and $\wt u_k^0$, $\wt u_k$
be as defined above. 
Then the limit
\begin{equation}
T_E
=
\lim_{k\to\infty}
\left\{
\int_{\wt A}\wt u_k(y)  \frac{\wt \omega^{n-1}(y)}{(n-1)!}
-
\int_{\wt A_0}\wt u_k^0(y)  \frac{\wt \omega_0^{n-1}(y)}{(n-1)!}
\right\}
\label{e6}
\end{equation}
exists and is independent of the choice of the sequence $\{\Omega_k\}$.
If, in addition, Assumption~\ref{ass2} holds true and $\tau_E:\Gamma\to\R$ is 
as defined by \eqref{a13}, then 
\begin{equation}
T_E=\int_{\Gamma} \tau_E(x)\frac{\omega^{n-1}(x)}{(n-1)!}.
\label{e7}
\end{equation}
\end{theorem}
$T_E$ is the \emph{total time delay}. 
The above statement (in various concrete forms) is well known. 
See \cite{Nuss} for a survey of time delay
and \cite{Robert1} for a rigourous discussion of semiclassical aspects. 

\begin{remark}
Since $\wt W_-:\wt A_0\to \wt A$ (see \eqref{a12}) is a symplectic diffeomorphism, 
the difference of the integrals in \eqref{e6} can be rewritten as 
\begin{equation}
\int_{\wt A}\wt u_k(y)  \frac{\wt \omega^{n-1}(y)}{(n-1)!}
-
\int_{\wt A_0}\wt u_k^0(y)  \frac{\wt \omega_0^{n-1}(y)}{(n-1)!}
=
\int_{\wt A_0}
(\wt u_k\circ\wt W_-(y)-\wt u_k^0(y))  
\frac{\wt \omega_0^{n-1}(y)}{(n-1)!}.
\label{e6a}
\end{equation}
The quantity $\wt u_k\circ\wt W_-(y)-\wt u_k^0(y)$ and its limit
\begin{equation}
\lim_{k\to\infty} (\wt u_k\circ\wt W_-(y)-\wt u_k^0(y))
\label{td}
\end{equation}
is often interpreted as time delay 
related to the orbit $y$. Note, however, that  the limit \eqref{td}
may not exist unless the sequence $\Omega_k$ is chosen in a special way; 
see \cite{Martin1} for a discussion of this issue.
\end{remark}

\subsection{Regularised phase space volume}
Suppose that Assumption~\ref{ass1}(iv) holds true for some $E$. Let us denote
\begin{equation}
\xi(E)
=
\int_{\N}(\chi_{G_0(E)}-\chi_{G(E)})\frac{\omega^n}{n!}.
\label{xi}
\end{equation}
It is interesting to note that
if $\pm (H(x)-H_0(x))\geq0$ for all $x\in \N$, then $\pm \xi(E)\geq0$. 
The ``quantum" analogue of $\xi(E)$ is the spectral shift function. 
See the survey \cite{Robert2} for an extensive discussion of this analogy in 
semiclassical context.

\section{Main results}\label{sec.d}

We will use some notation and terminology from symplectic topology. 
We collect the required material in Appendix~B; for the details, see \cite{McDuff}.
In particular, we use the notion of the Calabi invariant. 
For a compactly supported symplectomorphism $\psi:\M\to \M$ 
of an exact non-compact symplectic manifold $\M$,
one defines the Calabi invariant  $\CAL(\psi)$ as  
the appropriately normalised integral of the generating function of $\psi$. 
The Calabi invariant is defined for those compactly supported symplectomorphisms $\psi$
that have a compactly supported generating function; we denote this set 
of symplectomorphisms by $\Dom(\CAL, \M)$. 
We note that our sign conventions and normalisation of  the Calabi 
invariant are different from those of \cite{McDuff}.

\subsection{$\CAL(\wt S_E)$ and the regularised phase space volume}\label{sec.3a}
\begin{theorem}\label{maintheorem}
Let $n\geq2$; 
suppose  that Assumption~\ref{ass1} and  completeness \eqref{completeness} hold true
for some $E\in\R$.
Assume also that $\N$ is exact and that $\wt A_0(E)$ is non-compact.
Then the map $\wt S_E:\wt A_0\to\wt A_0$ 
belongs to $\Dom(\CAL, \wt A_0)$ and the identity
\begin{equation}
\CAL(\wt S_E)=\xi(E)
\label{e1}
\end{equation}
holds true.
\end{theorem}
This should be compared to the 
Birman-Krein formula \eqref{BK} in ``quantum" scattering theory, 
bearing in mind the analogy between the Calabi invariant 
and the logarithm of determinant, see Section~\ref{sec.app3}.

\subsection{The scattering matrix and the total time delay}\label{sec.3b}

\begin{theorem}\label{th2}
Suppose that Assumption~\ref{ass1} and completeness \eqref{completeness}
hold true for some $E\in\R$. 
Suppose also that Assumption~\ref{ass1}(iv) holds true for some $E_1>E$. 
Then the identity
\begin{equation}
T_E=-\frac{d\xi}{dE}(E)
\label{e4a}
\end{equation}
holds true.
\end{theorem}
This result in various concrete forms appeared before in physics literature; see e.g. \cite{NT, BO,LV}.
Combining Theorem~\ref{th2} with Theorem~\ref{maintheorem}, we get
$$
\frac{d}{dE}\CAL(\wt S_E)=-T_E.
$$
This should be compared to the 
Eisenbud-Wigner formula \eqref{EW} in ``quantum'' scattering.
A related result was obtained in \cite{Bolle} in the framework of Hilbert space classical scattering.

\subsection{$\wt S_E$ as a Hamiltonian symplectomorphism }
Recall the inclusions \eqref{e8}. It is easy to show (see Example~\ref{exa.6})
that $\wt S_E$ may fail to belong to $\Symp_0^c(\wt A_0)$, 
i.e. it may not be possible to continuously deform $\wt S_E$ into the identity map.
However, the following theorem shows that under some additional 
assumptions, $\wt S_E\in \Ham^c(\wt A_0)\subset \Symp_0^c(\wt A_0)$. 

\begin{theorem}\label{th3}
Assume that there exists a smooth family of Hamiltonians $H_s$, $s\in [0,1]$ such that 
$H_1=H$. Suppose that for all $s\in[0,1]$, Assumption~\ref{ass1} and \eqref{completeness} 
hold true for the pair of Hamiltonians $H_0$, $H_s$. Suppose also that the set
$A_0(E)\cap\bigl(\cup_{s\in[0,1]}\supp(H_s-H_0)\bigr)$ is pre-compact. 
Then the scattering symplectomorphism $\wt S_E$ (corresponding 
to the pair $H_0$, $H_1$) is a compactly supported Hamiltonian symplectomorphism: 
$\wt S_E \in \Ham^c(\wt A_0)$. 
\end{theorem}
An explicit formula for a family of Hamiltonians which generate $\wt S_E$ is given in Section~\ref{sec.i}, 
see \eqref{i16}, \eqref{i17}.

\section{Example: classical mechanics}\label{sec.k}
\subsection{The general case}
Let $\N=\R^{2n}$, $n\geq2$, with the standard symplectic form $\omega=\sum_{i=1}^n dp_i\wedge dq_i$, 
$(q_1\dots q_n,p_1\dots p_n)=x\in\R^{2n}$. Of course, $\N$ is exact, 
$\omega=d(-\sum_i q_i dp_i)$. 
We denote by $\langle\cdot,\cdot\rangle$ the usual inner product in $\R^n$, and 
$\abs{q}^2=\langle q,q\rangle$. 
Let 
$$
H_0(q,p)=\frac12\abs{p}^2+v_0(q), 
$$
where $v_0\in C^\infty(\R^n)$ satisfies the following assumptions: 
\begin{gather}
\sup_{q\in\R^n}(v_0(q)+\frac12 \langle q,\nabla v_0(q)\rangle)<\infty,
\label{k1}
\\
\inf_{q\in\R^n} v_0(q)>-\infty.
\label{k2}
\end{gather}
The quantity $v_0+\frac12\langle q,\nabla v_0\rangle$ in \eqref{k1} is known as 
the virial. 
Next, let $H(q,p)=H_0(q,p)+v(q)$, where $v\in C_0^\infty(\R^n)$. Let $E\in\R$
be such that
\begin{gather}
\text{$E$ is not a critical value of $v_0$ or $v_0+v$;}
\label{k3}
\\
E>\sup_{q\in\R^n}(v_0(q)+\frac12\langle q,\nabla v_0(q)\rangle).
\label{k4}
\end{gather}
Fix any $R\in\R$ and consider 
\begin{equation}
\Gamma=\{(q,p)\mid h_0(q,p)=E, \langle q,p\rangle=R\}\subset A_0(E).
\label{k5}
\end{equation}
\begin{lemma}\label{lma.k1}
Assume \eqref{k1} through \eqref{k4}. Then Assumption~\ref{ass1} 
and Assumption~\ref{ass2} hold true with $\Gamma$ as in 
\eqref{k5}. 
\end{lemma}
\begin{proof} 
Assumption~\ref{ass1}(i) follows from \eqref{k3}. 
Assumption~\ref{ass1}(ii) follows from \eqref{k2}, since
$\dot q=p$ and $\abs{p}^2=2(E-v_0)\leq C<\infty$.  

Let us check Assumption~\ref{ass1}(iv). We have
$$
G_0(E)\cap \supp(H-H_0)=\{(q,p)\mid q\in \supp v, \frac12 \abs{p}^2\leq E-v_0(q)\};
$$ 
using \eqref{k2} we see that this set is compact. In the same way, 
$G(E)\cap \supp(H-H_0)$ is compact. 

Let us check Assumption~\ref{ass2}. In order to check that $\Gamma$ is a smooth 
manifold in $A_0(E)$, it suffices to verify that $dH_0$ and $d\langle q,p\rangle$ are linearly independent
on $\Gamma$. Suppose that $dH_0=\lambda d\langle q,p\rangle$ at some point $(q,p)\in\Gamma$. 
Then $\nabla v_0(q)=\lambda p$ and $p=\lambda q$, and so 
$E=\frac12 \abs{p}^2+v_0(q)=\frac12\langle \nabla v_0(q),q\rangle+v_0(q)<E$
by \eqref{k4} --- contradiction. 

Next, we need to check that $X_0$ is non-tangential to $\Gamma$. We have
$X_0(q,p)=(p,-\nabla v_0(q))$. The tangent space $T_{(q,p)}\Gamma$ consists
of vectors $(\xi,\eta)$ such that $\langle \xi,p\rangle+\langle \eta,q\rangle =0$
and $\langle \xi,\nabla v_0(q)\rangle +\langle \eta,p\rangle=0$. 
For $(\xi,\eta)=X_0(q,p)$ we have
$$
\langle \xi,p\rangle+\langle \eta,q\rangle=\abs{p}^2-\langle q,\nabla v_0(q)\rangle=2(E-v_0(q))-\langle q,\nabla v_0(q)\rangle >0
$$
by \eqref{k4}, and so $X_0(q,p)$ is not in  $T_{(q,p)}\Gamma$. 

Finally, for any trajectory $(q(t),p(t))=\Phi^0_t(x)$, we have
\begin{equation}
\frac{d}{dt}\langle q(t), p(t)\rangle =\abs{p(t)}^2-\langle q(t),\nabla v_0(q(t))\rangle\geq C>0
\label{k6}
\end{equation}
by \eqref{k4}, and so there exists a unique $t\in\R$ such that $\langle q(t), p(t)\rangle =R$. 

Assumption~\ref{ass2}, together with Assumption~\ref{ass1}(i),(ii), implies Assumption~\ref{ass1}(iii). 
\end{proof}

Thus, \eqref{k1}--\eqref{k4} ensure that the wave operators $W_\pm:A_0(E)\to A(E)$ 
exist. In order to ensure that completeness \eqref{completeness} holds true, 
we have to make more specific assumptions. 
Let us  assume that 
\begin{equation}
E>\sup_{q\in\R^n}(v_0(q)+v(q)+\tfrac12\langle q,\nabla v_0(q)\rangle+\tfrac12\langle q,\nabla v(q)\rangle).
\label{k7}
\end{equation}
Then by the same argument as in the proof of Lemma~\ref{lma.k1} (see \eqref{k6}), we obtain that 
all trajectories of $\Phi_t$ leave $\supp v$ for large $\abs{t}$ and therefore coincide with some 
trajectories of the free dynamics for large $\pm t>0$. From here we get completeness \eqref{completeness}. 

It is also clear that the family $H_s(q,p)=H_0(q,p)+sv(q)$, $s\in[0,1]$,  satisfies assumptions of Theorem~\ref{th3}
for $E$ as in \eqref{k7}.
Thus, in this case $\wt S_E$ is a compactly supported Hamiltonian symplectomorphism. 

To summarise:  Suppose that  \eqref{k1}--\eqref{k4}, \eqref{k7} hold true. Then Assumptions~\ref{ass1},
\eqref{completeness} and Assumption~\ref{ass2} 
hold true and therefore Theorems~\ref{maintheorem}, \ref{th2}, \ref{th3} hold true.

\subsection{The case $v_0\equiv0$}
In the case $v_0\equiv0$
it is easy to give explicit formulas for the phase space volume, time delay, and the Calabi invariant of $\wt S_E$. 

The formula for the regularised phase space volume \eqref{xi}  is very well known (see e.g. \cite{Robert2} 
for a comprehensive discussion) and easy to compute:  
\begin{equation}
\xi(E)
=
\varkappa_n(2E)^{n/2}\int_{\R^n} \bigl(1-(1-v(q)E^{-1})_+^{n/2}\bigr)d^nq, \quad E>0,
\label{k10}
\end{equation}
where $\varkappa_n=\frac{\pi^{n/2}}{\Gamma(1+n/2)}$ is the volume of a unit ball in $\R^n$,
 $d^n q$ is the Lebesgue measure in $\R^n$, and $(a)_+=(a+\abs{a})/2$.

Next, the formula for the time delay is also well known; see e.g. \cite{Narnhofer}. 
Let $x^-=(q^-,p^-)\in\Gamma$ and let $x^+=(q^+,p^+)=S_E(x^-)$. 
Then the time delay function $\tau_E$ from \eqref{a13} can be expressed as 
$$
\tau_E(x^-)=\frac1{2E}(\langle q^-,p^-\rangle-\langle q^+,p^+\rangle),
$$
and the total time delay is 
$$
T_E=\int_\Gamma \tau_E(x^-)\frac{\omega^{n-1}(x^-)}{(n-1)!} 
=
(2E)^{(n-3)/2}\int_{{\mathbb S}^{n-1}}d^{n-1}\hat p^- \int_{\langle q^-,p^-\rangle=R} d^{n-1}q^- 
\  (\langle q^-,p^-\rangle-\langle q^+,p^+\rangle),
$$
where 
$\hat p^-=p^-/\abs{p^-}$, and $d^{n-1}\hat p^-$ is the Lebesgue measure on the unit sphere
in $\R^n$. 

Let us display formula \eqref{e4a} in the cases $n=2,3$:
\begin{align*}
T_E&=0 \qquad &(n=2),
\\
T_E&=6\pi \sqrt{2E}
\int_{\R^3} \bigl(1-(1-v(q)E^{-1})^{1/2}\bigr)d^3q \qquad &(n=3).
\end{align*} 

Finally, let us give a formula for $\CAL (\wt S_E)$. This requires some preliminaries.
Let us fix $\alpha=-\sum_{i=1}^n q_idp_i$. Note that the form $\alpha$ is invariant with respect to 
simultaneous rotations of coordinate axes in both configuration and momentum space:
$q\mapsto Oq$, $p\mapsto Op$, $O$ is an orthogonal matrix. 
Let us describe local symplectic coordinates on $\Gamma$. 
For $x^-=(q^-,p^-)\in\Gamma$, let the set $U_{p^-}\subset \Gamma$ be 
$$
U_{p^-}=\{(q,p)\in \Gamma \mid \abs{p-p^-}<\sqrt{2E}\}.
$$
Let us simultaneously rotate the coordinate axes in both configuration and momentum space
in such a way that after this rotation, we have  $p^-=(0,\dots,0,\sqrt{2E})$. 
Then $(q_1,\dots,q_{n-1},p_1,\dots,p_{n-1})$ are symplectic coordinates on $U_{p^-}$. 

Recall the representation \eqref{a13} for the scattering map $S_E$. 
Fix $x^{-}=(q^{-},p^{-})\in\Gamma$ and let 
$\wt x=(\wt q,\wt p)=\wt s_E(x^{-})\in\Gamma$.
Let $(q_1,\dots,q_{n-1},p_1,\dots,p_{n-1})$ be symplectic coordinates in $U_{p^{-}}$ and 
let $(q'_1,\dots,q'_{n-1},p'_1,\dots,p'_{n-1})$ be symplectic coordinates in $U_{\wt p}$. 
Then the derivatives $\frac{\partial p'_i}{\partial q_j}$, $i,j=1,\dots,n-1$ are well defined. 
Define
$$
\rho(x^{-})
=
-\sum_{i,j=1}^{n-1} q_i^{-} \wt q_j \frac{\partial p'_j}{\partial q_i}(q^{-},p^{-}).
$$
We shall consider 
$\rho$ as a real valued function of $x^{-}\in \Gamma$. 
It is not difficult to see directly that this definition is independent of the choice of the local 
symplectic coordinates near $x^{-}$ and $\wt x$.
This will also follow from the proof of 
\begin{lemma}
One has
\begin{align}
\CAL(\wt S_E)
&=
\frac1{n(n-1)} \int_\Gamma \rho(x)\frac{\omega^{n-1}(x)}{(n-1)!}
\label{k8}
\\
&=
\frac{(2E)^{(n-1)/2}}{n(n-1)}\int_{{\mathbb S}^{n-1}}d^{n-1}\hat p \int_{\langle q,p\rangle=R} d^{n-1}q \ \rho(q,p).
\label{k8a}
\end{align}
\end{lemma}
\begin{proof} 
According to formula \eqref{e5} for  $\CAL$, we have 
$$
\CAL(\wt S_E)
=
\CAL(\wt s_E)
=
\frac1{n!} \int_{\Gamma} (\tilde s_E^* \alpha)\wedge \alpha\wedge \omega^{n-2},
$$
with $\alpha=-\sum_i q_idp_i$.
As above, let $\wt x=\wt s_E(x^{-})$, let $(q_1,\dots,q_{n-1},p_1,\dots,p_{n-1})$ be symplectic coordinates in $U_{p^{-}}$ and 
let $(q'_1,\dots,q'_{n-1},p'_1,\dots,p'_{n-1})$ be symplectic coordinates in $U_{\wt p}$. 
In these coordinate systems, let us write the form $\alpha$ on $\Gamma$ as 
$$
\alpha=-\sum_{i=1}^{n-1}\alpha_i dp_i, 
\quad
\tilde s_E^* \alpha=-\sum_{i=1}^{n-1}\alpha'_i dp'_i.
$$
Note that $\alpha_i(x^{-})=q_i^{-}$ and $\alpha'_i(\wt x)=\wt q_i$, $i=1,\dots, n-1$. 
Then it is easy to compute
\begin{multline*}
(\tilde s_E^* \alpha)\wedge \alpha\wedge \omega^{n-2}
=
(\sum_{i=1}^{n-1}\alpha'_i dp'_i)\wedge
(\sum_{j=1}^{n-1}\alpha_j dp_j)\wedge \omega^{n-2}
\\
=
(\sum_{k,i=1}^{n-1} \alpha_i'\frac{\partial p_i'}{\partial q_k} dq_k
+
\sum_{k,i=1}^{n-1}\alpha_i'\frac{\partial p_i'}{\partial p_k}dp_k)
\wedge (\sum_{j=1}^{n-1}\alpha_j dp_j)\wedge \omega^{n-2}
\\
=
(\sum_{k,i=1}^{n-1} \alpha_i'\frac{\partial p_i'}{\partial q_k} dq_k)
\wedge (\sum_{j=1}^{n-1}\alpha_j dp_j)\wedge \omega^{n-2}
=
\frac1{n-1}\sum_{i,j=1}^{n-1}\alpha'_i\alpha_j \frac{\partial p_i'}{\partial q_j} 
dq_j\wedge dp_j\wedge
\omega^{n-2}
\\
=
-\frac1{n-1}\sum_{i,j=1}^{n-1}\alpha'_i\alpha_j \frac{\partial p_i'}{\partial q_j} 
\omega^{n-1}
=
\frac1{n-1}\rho \omega^{n-1},
\end{multline*}
and \eqref{k8} follows. Formula \eqref{k8a} is just \eqref{k8} with $\omega^{n-1}$ expanded.
\end{proof}

\subsection{The case of rotationally symmetric $v$}
Let $v_0\equiv0$ and let  $v$ be rotationally symmetric, $v(x)=v_1(\abs{x})$.  Then formula \eqref{k8} can be recast in 
terms of the usual variables of scattering theory: the impact parameter $s$ and the scattering angle $\phi$
(see \cite[Section~18]{LL}). 
For $(q,p)\in\R^{2n}$, $p\not=0$, the impact parameter $s>0$ is defined by 
$$
s=\aabs{q-\frac{\langle q,p\rangle p}{\abs{p}^2}}.
$$
Due to the conservation of angular momentum, the impact parameter is the integral of motion 
for both dynamics $\Phi^0$, $\Phi$. 
If $(q^+,p^+)=S_E(q^-,p^-)$, then the scattering angle $\phi$ is defined such that 
\begin{equation}
\cos\phi=
\frac{\langle p^+,p^-\rangle}{\abs{p^+}^2}.
\label{scangle}
\end{equation}
Of course, this does not yet define $\phi$ uniquely. 
In order to fix $\phi$, let us note that due to the rotational symmetry, $\phi$ depends only on $E$ and $s$. 
Thus, for a fixed energy $E$ let us define $\phi$ as a continuous function $\phi:(0,\infty)\to\R$
such that \eqref{scangle} holds true, $\phi(s)\to0$ as $s\to\infty$, 
and $\pm\phi(s)\geq0$ for large $s$ if $\pm\langle p_+,q_-\rangle\geq0$
(here for simplicity we assume $R=0$ in \eqref{k4}).
Formula for $\phi$ is well known (see e.g. \cite[Section~18]{LL}):
\begin{equation}
\phi(s)=\pi-2\int_{r_{min}}^\infty \frac{sdr}{r^2} \biggl(1-\frac{v_1(r)}{E}-\frac{s^2}{r^2}\biggr)^{-1/2}, 
\quad
r_{min}(s)=\max\left\{r: 1-\frac{v_1(r)}{E}-\frac{s^2}{r^2}=0 \right\}.
\label{angle}
\end{equation}
It is easy to compute that 
$$
\rho=\rho(s)=-\sqrt{2E}s^2\frac{d\phi}{ds},
$$
and therefore 
\begin{equation}
\CAL(\wt S_E)
=
-\varkappa_n\varkappa_{n-1} (2E)^{n/2} \int_0^\infty s^n \frac{d\phi}{ds}ds
=
n\varkappa_n\varkappa_{n-1} (2E)^{n/2} \int_0^\infty s^{n-1}  \phi(s)ds.
\label{k9}
\end{equation}
Substituting \eqref{angle} into \eqref{k9}, and using \eqref{k10}, it is not difficult to check directly the validity of Theorem~\ref{maintheorem} 
in this case.

\begin{example}\label{exa.6}
Let us give an example where the scattering symplectomorphism is not 
homotopic to the identity map.
Let $n=2$, $v_0=0$ and $v$ be of the form
$v(q)=v_1(\abs{q})$, $v_1\geq0$, $v_1(r)=0$ for $r\geq1$, 
$v_1'(r)<0$ for all $r\in(0,1)$. Let $E\in(0,v_1(0))$. 
Under these assumptions, \eqref{k7} may or may not hold true. 
However, using the separation of variables, one can directly 
check that the completeness condition \eqref{completeness}
holds true.

In order to make our notation more succinct, let us identify $\R^2$ with $\C$ in the usual way. 
Then the trajectories $q=q(t)$, $p=p(t)$ of the free dynamics $\Phi^0$ can be
parameterised by $\theta\in[0,2\pi)$ and $\sigma\in\R$ so that
$$
q(t)=i\sigma e^{i\theta}+t\sqrt{2E}e^{i\theta},
\quad 
p(t)=p(0)=\sqrt{2E}e^{i\theta}.
$$
Of course, $\abs{\sigma}$ is the impact parameter.

It is easy to see that $\sigma\sqrt{2E}$, $\theta$ are symplectic coordinates on $\wt A_0(E)$, 
and so $\wt A_0(E)$ can be identified with a cylinder $T^*\mathbb S^1$  (as in Example~\ref{exa.4}).
Due to the conservation of angular momentum, the scattering symplectomorphism has the form
$$
\wt S_E: (\sigma,\theta)\mapsto(\sigma,\theta+\varphi(\sigma))
$$
where $\varphi(\sigma)=0$ for $\sigma\leq-1$ and $\varphi(\sigma)=2\pi$
for $\sigma\geq1$. Here $\varphi$ is the scattering angle with a different normalisation. 
Since the map $\wt S_E$ ``twists'' the cylinder $T^*\mathbb S^1$, it is easy to see 
that $\wt S_E$ is not homotopic to the identity map: $\wt S_E\notin \Symp^c_0(T^*\mathbb S^1)$.
\end{example}

\section{Auxiliary statements}\label{sec.g}
\subsection{Exactness of $\wt A_0$}
Recall that $\N$ is called exact if there exists a 1-form $\alpha$ on $\N$ such that
$d\alpha=\omega$. 
Here we prove that exactness of $\N$ implies exactness of $\wt A_0$. 
In the course of the proof, we construct an atlas on $\wt A_0$; this construction
will be used in the proof of Lemma~\ref{lma.g2}.
\begin{lemma}\label{lma.g1}
Let Assumption~\ref{ass1} (i)--(iii) hold true and 
suppose that $\N$ is exact. 
Then there exists a 1-form $\alpha$ on $\N$ such that
$d\alpha=\omega$ and 
\begin{equation}
i(X_0)\alpha=0 \quad \text{ on } A_0.
\label{e0}
\end{equation} 
Moreover, there exists a 1-form $\wt \alpha$ on $\wt A_0$ such that $d\wt \alpha=\wt \omega_0$
and $\pi_0^*\wt \alpha=\alpha|_{A_0}$.
\end{lemma}

\begin{proof}

1. First let us construct an atlas on $\wt A_0$. 
Fix $y_0\in\wt A_0$ and choose $x_0\in y_0$. 
Using local coordinates around $x_0$, it is easy 
to construct a manifold $\Gamma_0\subset A_0$ 
of dimension $2n-2$ such that $x_0\in\Gamma_0$ and $X_0(x)\notin T_x\Gamma_0$ 
for all $x\in\Gamma_0$. By Assumption~\ref{ass1}(iii), 
there exists $T>0$ such that for all $\abs{t}\geq T$, one has
$\Phi_t^0(\overline\Gamma_0)\cap\overline\Gamma_0=\emptyset$; 
here $\overline\Gamma_0$ is the closure of $\Gamma_0$ in $A_0$. 
It follows that 
the orbit $\Phi^0_t(x_0)$ can intersect $\Gamma_0$ only finitely many times. 
By reducing $\Gamma_0$ if necessary, we can ensure that $\Phi^0_t(x_0)\notin\overline\Gamma_0$
for all $t\not=0$.

Next, since we may assume $\Gamma_0$ to be pre-compact, there exists $\epsilon>0$ such 
that for all $0<\abs{t}<\epsilon$, one has $\Phi_t^0(\overline \Gamma_0)\cap \overline\Gamma_0=\emptyset$.
Since the closed set $K=\{\Phi^0_t(x_0)\mid \epsilon\leq \abs{t}\leq T\}$ has empty
intersection with $\overline\Gamma_0$, there exists an open neighbourhood 
of $K$ which has empty intersection with $\overline \Gamma_0$. 
It follows that one can choose a subset $\Gamma\subset\Gamma_0$ (which is itself a manifold of dimension
$2n-2$) such that $\Phi_t^0(\Gamma)\cap\Gamma=\emptyset$ for all $t\not=0$.
Thus, we have 
constructed a ``local Poincar\'e section", i.e. 
$\Gamma$ parameterises 
$\{\Phi^0_t(x)\mid t\in\R, \  x\in\Gamma\}\subset A_0$
rather than the whole manifold $A_0$. 

Such a manifold $\Gamma$ can be constructed for any point $y_0\in\wt A_0$;
then the corresponding sets $\pi_0(\Gamma)$ form an open cover of $\wt A_0$. 
Let us choose a locally finite subcover $\{\pi_0(\Gamma_j)\}$ of this cover and a 
smooth partition of unity $\{\wt \zeta_j\}$ on $\wt A_0$ subordinate to this subcover
(i.e. for each $j$, $\supp \wt \zeta_j$ lies entirely within one of the sets $\pi_0(\Gamma_j)$).
Clearly, we have an associated partition of unity $\{\zeta_j\}$ 
on $A_0$, $\zeta_j=\wt \zeta_j\circ\pi_0$.

Of course, if Assumption~\ref{ass2} holds true, then the above atlas can be chosen 
to consist of just one map.

2. 
Since $\N$ is exact, there exists a 1-form $\beta$ on $\N$ such that $d\beta=\omega$. 
Let us construct $F\in C^\infty(\N)$ such that $\alpha=\beta+dF$ 
has the required properties. 
We will construct $F$ using the atlas described above. 

For each $j$, let us construct a function $F_j$ on $\Omega_j=\{\Phi^0_t(x)\mid t\in\R,\ x\in\Gamma_j\}$
such that $\beta(X_0)+dF_j(X_0)=0$ on $\Omega_j$. 
This can be done by setting $F_j=0$ on $\Gamma_j$ and then extending $F_j$ onto $\Omega_j$ by 
integrating the differential equation 
\begin{equation}
\frac{d}{dt} F_j\circ \Phi_t^0(x)=-\beta(X_0\circ \Phi_t^0(x))
\label{e4}
\end{equation}
along the orbits of $\Phi^0$.

Now let us define $F=\sum_j \zeta_j F_j$ on $A_0$ and extend $F$ onto the whole of $\N$
as a smooth function. Then on $A_0$ we have
\begin{equation}
i(X_0)\alpha=i(X_0)\beta+i(X_0)dF=
\sum_j\zeta_j \  (i(X_0)\beta+i(X_0)dF_j)+\sum_j F_j\   i(X_0)d\zeta_j.
\label{g1}
\end{equation}
The first sum in the r.h.s. of \eqref{g1} vanishes by the construction of $F_j$. 
Next, by the construction of $\zeta_j$ we have
$$
0=\frac{d}{dt}\zeta_j\circ \Phi_t^0(x)|_{t=0}=d\zeta_j(X_0(x)),
$$
and so the second sum in the r.h.s. of \eqref{g1} also vanishes. 
Thus, $\alpha$ satisfies \eqref{e0}.

3. Let us prove that there exists a 1-form $\wt \alpha$ on $\wt A_0$ such that 
$\pi_0^*\wt \alpha=\alpha$. 
First note that, by Cartan's formula for Lie derivative, 
$$
L_{X_0}\alpha=d i(X_0)\alpha+i(X_0)d\alpha
=0+i(X_0)\omega=-dH_0=0
$$
on $A_0$, where we have used \eqref{e0}.
It follows that 
\begin{equation}
(\Phi_t^0)^*\alpha |_{A_0}=\alpha |_{A_0}
\quad 
\text{ for all $t\in\R$.}
\label{c1}
\end{equation}
Next, let $x,y\in A_0$ and let $\xi\in T_x A_0$, $\eta\in T_y A_0$ be such that 
$\pi_0(x)=\pi_0(y)$ and $d_x\pi_0(\xi)=d_y\pi_0(\eta)$. 
This means that for some $t,c\in\R$, one has $y=\Phi_t^0(x)$ and 
$\eta=d_x\Phi_t^0(\xi)+c X_0(y)$. 
Then, using \eqref{e0} and \eqref{c1}, 
we obtain
$$
\alpha_y(\eta)=\alpha_y(d_x\Phi_t^0(\xi)+c X_0(y))=
\alpha_y(d_x \Phi_t^0(\xi))=\alpha_x(\xi).
$$
Thus, $\alpha: TA_0\to\R$ is constant on the pre-images of any point under the 
map $d\pi_0: TA_0\to T\wt A_0$. 
This shows that one can define a smooth 1-form $\wt \alpha$ on $A_0$ such that 
$\pi_0^*\wt \alpha=\alpha$. 

4. Let us prove that $d\wt \alpha=\wt \omega_0$. We have 
$$
\pi_0^*\wt \omega_0=\omega=d\alpha=d\pi_0^*\wt \alpha=\pi_0^*d\wt \alpha,
$$
and so $\pi_0^*(\wt \omega_0-d\wt \alpha)=0$. 
Since $\pi_0$ is a surjection, it follows that $d\wt \alpha=\wt \omega_0$. 
\end{proof}

\subsection{Separation of variables in integrals over $A_0$}
We will need the following version of Fubini's theorem:
\begin{lemma}\label{lma.g2}
Let $\mu$ be a 1-form on $A_0$ and $\Omega\subset A_0$ be an open pre-compact set.
Define 
$$
\delta(x)=\int_{-\infty}^\infty \chi_\Omega\circ \Phi_t^0(x)\  
(i(X_0)\mu)\circ\Phi_t^0(x)dt, 
\quad x\in A_0,
$$
and let $\wt \delta: \wt A_0\to \R$ be the corresponding function such that 
$\wt \delta\circ \pi_0=\delta$. 
Then 
$$
\int_{\Omega}\mu\wedge \omega^{n-1}
=
\int_{\wt A_0}\wt \delta \ 
\wt \omega_0^{n-1}.
$$
\end{lemma}
We note that this lemma, with obvious modifications, can (and will) also be 
applied to integrals over $A$ and $\wt A$ instead of $A_0$, $\wt A_0$.

\begin{proof}
1. Let $\zeta_j$, $\wt \zeta_j$, $\Gamma_j$ be as in the proof of Lemma~\ref{lma.g1}.
It suffices to prove that
\begin{equation}
\int_{\Omega}\zeta_j \mu\wedge \omega^{n-1}
=
\int_{\wt A_0}\wt \zeta_j\   \wt \delta \ \wt \omega_0^{n-1}.
\label{g1a}
\end{equation}
Next, as in Section~\ref{sec.a4}, we have a map $\gamma_j: \Gamma_j\to \wt A_0$
which is a symplectic diffeomorphism onto its range. 
Using this map, we can rewrite the integral in the r.h.s. of \eqref{g1a} 
as the integral over $\Gamma_j$. 
Thus, it suffices to prove that 
\begin{equation}
\int_{\Omega}\zeta_j \mu\wedge \omega^{n-1}
=
\int_{\Gamma_j} \zeta_j(x)  \omega^{n-1}(x)
\int_{-\infty}^\infty \chi_\Omega\circ \Phi_t^0(x)\  
(i(X_0)\mu)\circ\Phi_t^0(x)dt.
\label{g2}
\end{equation}

2. 
Let us prove  \eqref{g2}.
Consider the map $\phi:\R\times \Gamma_j\to A_0$, $\phi(t,x)=\Phi_t^0(x)$.
We have
\begin{equation}
\int_{\Omega}\zeta_j
\mu \wedge\omega^{n-1}
=
\int_{\R\times\Gamma_j}
(\zeta_j\circ \phi)(\chi_\Omega\circ \phi)
(\phi^*\mu)\wedge(\phi^*\omega)^{n-1}.
\label{c12a}
\end{equation}
Since $\Phi_t^0$ is a symplectic map, we get for any $\xi,\eta\in T_x\Gamma_j$:
$$
(\phi^*\omega)_{(t,x)}(\xi,\eta)
=
\omega_{\Phi^0_t(x)}(d_x\Phi^0_t(\xi),d_x\Phi^0_t(\eta))
=
\omega_x(\xi,\eta).
$$
Further, denote $m=2n-2$; let $(t,x_1,\dots,x_m)$ be local coordinates on $\R\times \Gamma_j$
and let $\frac{\partial}{\partial t},\frac{\partial}{\partial x_1},\dots,\frac{\partial}{\partial x_m}$
be the corresponding basis in the tangent space $T_{(t,x)}(\R\times\Gamma_j)$. 
For any $\xi\in T_x\Gamma_j$, we have
$$
(\phi^*\omega)_{(t,x)}(\frac{\partial}{\partial t},\xi)
=
\omega_x(X_0(x),d_x\Phi^0_t(\xi))
=
-d_x H_0(d_x \Phi^0_t(\xi))
=0,
$$
and so
$$
((\phi^*\mu)\wedge(\phi^*\omega)^{n-1})_{(t,x)}
(\frac{\partial}{\partial t},\frac{\partial}{\partial x_1},\dots, \frac{\partial}{\partial x_m})
=
\mu_{\Phi^0_t(x)}(X_0(\Phi^0_t(x)))\omega_x^{n-1}(\frac{\partial}{\partial x_1},\dots, \frac{\partial}{\partial x_m}).
$$
It follows that we can separate integration over $\Gamma_j$ and over $\R$ in \eqref{c12a}, 
which yields the required identity \eqref{g2}.
\end{proof}

\section{Proof of Theorem~\ref{maintheorem}}\label{sec.c}

\subsection{A generating function of $\wt S_E$} 
Let $\alpha$ be a 1-form on $\N$ as  in Lemma~\ref{lma.g1}.
Let us define a function $\Delta$ on $A$ by 
\begin{equation}
\Delta(x)=-\int_{-\infty}^\infty (i(X)\alpha)\circ \Phi_t(x)dt, 
\quad x\in A. 
\label{c2}
\end{equation}
Note that by \eqref{e0}, for all sufficiently large $\abs{t}$ one has
$$
(i(X)\alpha)\circ \Phi_t(x)=(i(X_0)\alpha)\circ \Phi_t(x)=0
$$
and so the integration in 
\eqref{c2} is actually performed over a bounded set of $t$. 
It follows that $\Delta\in C^\infty(A)$.

It follows directly from the definition that  $\Delta$ is constant
on the orbits of $\Phi_t$ and therefore there exists $\wt \Delta\in C^\infty(\wt A)$ 
such that $\wt \Delta\circ\pi=\Delta$. 
Since $A\cap\supp(H-H_0)$ is compact, it follows that $\supp \wt \Delta$ is
compact and so $\wt \Delta \in C_0^\infty(\wt A)$.

\begin{lemma}
Under the assumptions of Theorem~\ref{maintheorem}, we have
$\wt S_E\in\Dom(\CAL, \wt A_0)$ and 
\begin{equation}
n\CAL(\wt S_E)=\int_{\wt A}\wt \Delta (x)\frac{\wt \omega^{n-1}(x)}{(n-1)!}.
\label{c3}
\end{equation}
\end{lemma}
\begin{proof}

1. First recall the well known formula for the derivative of reduced action. 
Let $T>0$ and 
$$
f_T(x)=\int_0^T (i(X)\alpha)\circ \Phi_t(x) dt, \quad x\in A;
$$
then 
\begin{equation}
d f_T=\Phi_T^*\alpha-\alpha+TdH. 
\label{c5}
\end{equation}
On $A$, we have $H=E$ and so the last term in the r.h.s. of \eqref{c5} vanishes.

2. Let us check that the function $\Delta_0=\Delta\circ W_-\in C^\infty(A_0)$ 
satisfies the identity 
\begin{equation}
\alpha-S_E^*\alpha=d\Delta_0 
\quad \text{ on $A_0$.}
\label{c4}
\end{equation}

Let $K\subset A_0$ be a compact set. We can choose $T>0$ sufficiently large so that
for all $x\in K$, 
$$
S_E(x)=\Phi^0_{-T}\circ \Phi_{2T}\circ \Phi_{-T}^0(x)
$$
and 
$$
\Delta(x)
=
-\int_{-T}^T(i(X)\alpha)\circ \Phi_t(x)dt
=
-f_{2T}\circ \Phi_{-T}(x).
$$
Thus, using \eqref{c1} and \eqref{c5},  we have on $K$: 
\begin{multline*}
S_E^*\alpha
=
(\Phi^0_{-T})^* (\Phi_{2T})^*(\Phi^0_{-T})^*\alpha
=
(\Phi^0_{-T})^*(df_{2T}+\alpha)
\\
=
d(f_{2T}\circ\Phi^0_{-T})+\alpha
=
d(f_{2T}\circ \Phi_{-T}\circ W_-)+\alpha
=
-d(\Delta\circ W_-)+\alpha=-d\Delta_0+\alpha, 
\end{multline*}
which proves \eqref{c4}. 

3. Let $\wt \Delta_0\in C_0^\infty(\wt A_0)$ be a function such that 
$\wt \Delta_0\circ\pi_0=\Delta_0$. Using the formula 
$\alpha=\pi_0^*\wt \alpha$ and \eqref{a11}, we can rewrite \eqref{c4} as
$$
\pi_0^*(\wt \alpha-\wt S_E^*\wt \alpha)=\pi_0^*d\wt \Delta_0.
$$
Since $\pi_0$ is a surjection, it follows that 
\begin{equation}
\wt \alpha-\wt S_E^*\wt \alpha=d\wt \Delta_0 
\quad \text{ on $\wt A_0$, }
\label{c6}
\end{equation}
i.e. $\wt \Delta_0$ is a generating function of $\wt S_E$. 
Thus, $\wt S_E\in\Dom(\CAL, \wt A_0)$ and, according to the 
definition \eqref{e2},
\begin{equation}
n\CAL(\wt S_E)=\int_{\wt A_0}\wt \Delta_0(x)
\frac{\wt \omega_0^{n-1}(x)}{(n-1)!}.
\label{c7}
\end{equation}

4. 
Let $\wt W_-$ be as in \eqref{a12}. Since $\wt W_-^*\wt \omega=\wt \omega_0$, 
by a change of variable in the integral \eqref{c7} we obtain \eqref{c3}.
\end{proof}

\subsection{Application of Stokes' formula}
Let us apply Stokes' formula to rewrite the integral \eqref{xi} in the definition of $\xi$.
Let $\alpha$ be a 1-form on $\N$ as in Lemma~\ref{lma.g1}.
We first note that 
\begin{equation}
\alpha\wedge\omega^{n-1}|_{A_0}=0.
\label{c9}
\end{equation}
Indeed, $i(X_0)\alpha |_{A_0}=0$ by \eqref{e0}, and $i(X_0)\omega=-dH_0=0$ 
on $A_0$ since $A_0$ is a constant energy surface. Thus, 
$i(X_0)(\alpha\wedge\omega^{n-1})|_{A_0}=0$; 
since $\alpha\wedge\omega^{n-1}$ has a maximal rank on $A_0$,
it follows that \eqref{c9} holds true. 
\begin{lemma}\label{lma.c2}
The identity 
\begin{equation}
\xi(E)=-\frac1{n!}\int_{A(E)}\alpha\wedge \omega^{n-1}
\label{c10}
\end{equation}
holds true. 
\end{lemma}
\begin{proof}
1. Let $\Omega\subset \N$ be a compact set with a smooth boundary 
such that $\supp(H-H_0)\cap G(E)\subset\Omega$ and 
$\supp(H-H_0)\cap G_0(E)\subset\Omega$. Since 
$$
A(E)\setminus\Omega=A_0(E)\setminus\Omega, 
$$
by \eqref{c9} the integrand in \eqref{c10} vanishes outside $\Omega$. 
Thus, the integration in \eqref{c10} is in fact performed over $A(E)\cap \Omega$ 
and so the r.h.s. is finite. 

2. Writing $G=G(E)$ and $G_0=G_0(E)$
for brevity, we get
$$
\xi(E)
=
\int_{\N}(\chi_{G_0}-\chi_G)\frac{\omega^n}{n!}
=
\int_{\N\cap\Omega}(\chi_{G_0}-\chi_G)\frac{\omega^n}{n!}
=
\int_{G_0\cap\Omega}\frac{\omega^n}{n!}
-
\int_{G\cap\Omega}\frac{\omega^n}{n!}.
$$
We have $d(\alpha\wedge\omega^{n-1})=\omega^n$
and therefore, by the Stokes' formula, 
\begin{multline*}
n! \xi(E)
=
\int_{\partial (G_0\cap \Omega)}\alpha\wedge \omega^{n-1}
-
\int_{\partial (G\cap \Omega)}\alpha\wedge \omega^{n-1}
\\
=
\int_{A_0(E)\cap \Omega}\alpha\wedge \omega^{n-1}
-
\int_{A(E)\cap \Omega}\alpha\wedge \omega^{n-1}
=
-\int_{A(E)}\alpha\wedge \omega^{n-1}
\end{multline*}
since the integrals over 
$G_0\cap\partial {\Omega}$ and $G\cap\partial{\Omega}$ 
cancel out. 
\end{proof}

\subsection{The rest of the proof}\label{sec.c4}

By \eqref{c3} and \eqref{c10}, it remains to prove that 
$$
-\int_{A}\alpha\wedge \omega^{n-1}
=
\int_{\wt A}\wt \Delta(x)\wt \omega^{n-1}(x).
$$
This follows from Lemma~\ref{lma.g2} with $\mu=\alpha$, $\Omega=A$.

\section{Time delay: Proof of Theorems~\ref{th4} and \ref{th2} }\label{sec.6}

\subsection{Proof of Theorem~\ref{th4}}

1. Our first aim is to prove that the r.h.s. of \eqref{e6a} is independent of $k$ for all sufficiently large $k$. 
Denote $K_0=\supp(H-H_0)\cap A_0$ and $\wt K_0=\pi_0(K_0)$.
It is easy to see that $\supp(\wt u_k\circ\wt W_- -\wt u_k^0)\subset \wt K_0$ 
and therefore the integration in the r.h.s. of \eqref{e6a} is actually performed over $\wt K_0$. 

By Assumption~\ref{ass1}(iii), there exists $T>0$ such that 
for all $\abs{t}\geq T$ one has
\begin{equation}
\Phi_t^0(K_0)\cap K_0=\emptyset.
\label{i22}
\end{equation}
Then we have 
\begin{align*}
\Phi_t\circ W_-(x)&=\Phi_t^0(x), \quad t\leq -T, \quad \forall x\in K_0,
\\
\Phi_t\circ W_-(x)&=\Phi_t^0\circ S_E(x), \quad t\geq T \quad \forall x\in K_0.
\end{align*}
Let us choose $\ell$ sufficiently large so that for all $\abs{t}\leq T$, we have
\begin{equation}
\Phi_t^0(K_0)\subset \Omega_\ell, 
\quad
\Phi_t\circ W_-(K_0)\subset \Omega_\ell.
\label{i23}
\end{equation}
Then for all $x\in K_0$ and all $k\geq \ell$: 
\begin{equation}
u_k\circ W_-(x)-u_k^0(x)
=
\int_0^\infty (\chi_{\Omega_k}\circ \Phi_t^0\circ S_E(x)-\chi_{\Omega_k}\circ \Phi_t^0(x))dt. 
\label{i5}
\end{equation}
Next, let us define
\begin{equation}
v_k(x)=\int_0^\infty \chi_{\Omega_k\setminus\Omega_\ell}\circ \Phi_t^0(x)dt, 
\quad x\in K_0, \quad k\geq \ell. 
\label{i6}
\end{equation}
Comparing \eqref{i5} and \eqref{i6}, we get
\begin{equation}
(u_k\circ W_-(x)-u_k^0(x))
-
(u_\ell\circ W_-(x)-u_\ell^0(x))
=
v_k\circ S_E(x)-v_k(x), \quad x\in K_0. 
\label{i24}
\end{equation}
From  \eqref{i23} it follows that $v_k\circ \Phi_t^0(x)=v_k(x)$ for all $x\in K_0$ and $\abs{t}\leq T$. 
Let us define $\wt v_k: \wt K_0\to\R$ such that $\wt v_k\circ \pi_0(x)=v_k(x)$, 
$x\in K_0$.
Then from \eqref{i24} it follows that 
\begin{multline}
\int_{\wt K_0}
\left\{
(\wt u_k\circ \wt W_-(y)-\wt u_k^0(y))
-
(\wt u_\ell\circ \wt W_-(y)-\wt u_\ell^0(y))
\right\}
\frac{\omega^{n-1}_0(y)}{(n-1)!}
\\
=
\int_{\wt K_0}(\wt v_k\circ \wt S_E(y)-\wt v_k(y))\frac{\wt \omega_0^{n-1}(y)}{(n-1)!}.
\label{i7}
\end{multline}
Since $\wt S_E: \wt A_0\to\wt A_0$ is a symplectic diffeomorphism and $\wt S_E(x)=x$ for $x\notin\wt K_0$, 
we see that the r.h.s. of \eqref{i7} vanishes. This proves that  the r.h.s. of \eqref{e6a} is independent of $k\geq \ell$.
Thus, the limit \eqref{e6} exists.

2. Let us prove that the limit in \eqref{e6} is independent of the choice of the sequence
$\{\Omega_k\}$. 
First note that the limit \eqref{e6} can be calculated over any subsequence of $\{\Omega_k\}$.
Next, let $\{\Omega_k'\}$ be another sequence of sets with the same properties as $\{\Omega_k\}$.
Then it is easy to construct sequences of indices $p_1<p_2<\cdots$ and $q_1<q_2<\cdots$ 
such that 
\begin{equation}
\Omega_{p_1}\subset \Omega'_{q_1}\subset \Omega_{p_2}\subset\Omega'_{q_2}\subset\cdots
\label{i7a}
\end{equation}
Then the sequence $\{\Omega_{p_k}\}$ is a subsequence of both the sequence \eqref{i7a}
and the sequence $\{\Omega_k\}$. 
It follows that the limits \eqref{e6} over the sequence \eqref{i7a} and over the sequence $\{\Omega_k\}$ 
coincide. 
In the same way, the limits \eqref{e6} over the sequence \eqref{i7a} and over the sequence $\{\Omega_k'\}$ 
coincide. 

3. Let us prove \eqref{e7}. Since $\supp(H-H_0)\cap A_0$ is compact, there exists a compact set 
$\Gamma_0\subset \Gamma$ such that for all $z\in\Gamma\setminus \Gamma_0$ and all $t\in\R$, 
one has $\Phi_t(z)=\Phi_t^0(z)$ (for example, one can take $\Gamma_0=\Gamma\cap K_0$).
Then we have $\supp\tau_E\subset \Gamma_0$ and also 
$$
T_E=\int_{\Gamma_0} (u_k\circ W_-(z)-u_k^0(z))\frac{\omega^{n-1}(z)}{(n-1)!}, 
$$
for all sufficiently large $k$. 

Let $\ell$ be sufficiently large so that \eqref{i23} holds true and also assume 
(by increasing $\ell$ if necessary) that for all $k\geq \ell$ and all $z\in\Gamma_0$, one has
$\Phi_t^0(z)\in\Omega_k$ for $\abs{t}\leq \abs{\tau(z)}$.
Using \eqref{i5} and  the representation \eqref{a13} for the scattering map, we get for all $z\in\Gamma_0$: 
\begin{multline}
u_k\circ W_-(z)-u_k^0(z)
=
\int_0^\infty \{\chi_{\Omega_k}\circ\Phi_t^0\circ \Phi^0_{-\tau(z)}\circ \wt s_E(z)
-\chi_{\Omega_k}\circ \Phi_t^0(z)\}dt
\\
=
\int_0^\infty \{\chi_{\Omega_k}\circ\Phi_t^0\circ \wt s_E(z)
-\chi_{\Omega_k}\circ \Phi_t^0(z)\}dt
+
\int_{-\tau(z)}^0 \chi_{\Omega_k}\circ\Phi_t^0\circ \wt s_E(z)dt
\\
=
\int_0^\infty \{\chi_{\Omega_k}\circ\Phi_t^0\circ \wt s_E(z)
-\chi_{\Omega_k}\circ \Phi_t^0(z)\}dt
+\tau(z).
\label{i8}
\end{multline}
Next, since $\wt s_E:\Gamma\to\Gamma$ is a symplectic diffeomorphism, we get
\begin{equation}
\int_{\Gamma_0}\chi_{\Omega_k}\circ \Phi_t^0\circ \wt s_E(z) \omega^{n-1}(z)
=
\int_{\Gamma_0}\chi_{\Omega_k}\circ \Phi_t^0 (z) \omega^{n-1}(z)
\label{i9}
\end{equation}
for all $t\in\R$. Finally, integrating \eqref{i8} over $\Gamma_0$ with respect to 
the symplectic volume form
$\frac{\omega^{n-1}}{(n-1)!}$ and using \eqref{i9}, we arrive at \eqref{e7}.

\subsection{Proof of Theorem~\ref{th2}}
1.
Let $\Omega_1\subset \Omega_2\subset\dots\subset \N$ be a sequence of open 
pre-compact sets such that $\cup_{k=1}^\infty \Omega_k=\N$. 
Choose $\ell$ sufficiently large so that 
$$
G(E_1)\cap\supp(H-H_0)\subset\Omega_\ell
\quad\text{ and }
G_0(E_1)\cap\supp(H-H_0)\subset\Omega_\ell.
$$
Then for all $k\geq\ell$, 
$$
\frac{d\xi}{dE}(E)
=
\frac{d}{dE}
\left\{
\int_{G_0(E)\cap \Omega_k}\frac{\omega^n}{n!}
-
\int_{G(E)\cap \Omega_k}\frac{\omega^n}{n!}
\right\}.
$$

2.
Let $Y_0$ be a vector field defined in a neighbourhood of $A_0(E)\cap\Omega_k$
such that 
$dH_0(Y_0) |_{A_0(E)}=1$.
Similarly, let $Y$ be a vector
field defined in a neighbourhood of $A(E)\cap\Omega_k$ such that 
$dH(Y) |_{A(E)}=1$.
Then 
\begin{multline*}
\frac{d\xi}{dE}(E)
=
\frac1{n!}
\int_{A_0(E)\cap\Omega_k}i(Y_0)\omega^n
-
\frac1{n!}
\int_{A(E)\cap\Omega_k}i(Y)\omega^n
\\
=
\frac1{(n-1)!}
\int_{A_0(E)\cap\Omega_k}(i(Y_0)\omega)\wedge\omega^{n-1}
-
\frac1{(n-1)!}
\int_{A(E)\cap\Omega_k}(i(Y)\omega)\wedge\omega^{n-1}.
\end{multline*}

3.
Let us apply Lemma~\ref{lma.g2} to the integrals in the r.h.s. of the last
identity.
We will take $\mu=i(Y_0)\omega$ for the first integral and $\mu=i(Y)\omega$
for the second integral.
Note that
$i(X)i(Y)\omega=i(Y)dH=1$
by our choice of $Y$, and in the same way $i(X_0)i(Y_0)\omega=1$.
Application of Lemma~\ref{lma.g2} yields
$$
\frac{d\xi}{dE}(E)
=
\int_{\wt A_0}\wt u^0_k(y)  \frac{\wt \omega_0^{n-1}(y)}{(n-1)!}
-
\int_{\wt A}\wt u_k(y)  \frac{\wt \omega^{n-1}(y)}{(n-1)!}
$$
with $\wt u_k^0$, $\wt u_k$ as in Section~\ref{sec.ttd}. 
By \eqref{e6}, the r.h.s. is $(-T_E)$, as required.

\section{Proof of Theorem~\ref{th3}}\label{sec.i}
1. Let $\Phi^s_t$ be the Hamiltonian flow of $H_s$.
For each $s$, one can consider the constant energy surface $A_s=A_s(E)$; 
we have the symplectic reduction procedure which yields the manifold $\wt A_s$ of the 
orbits of $\Phi^s_t$. Let $\pi_s:A_s\to\wt A_s$ be the natural projection. Finally, let 
 $W_\pm^s$, $S^s_E$, $\wt S^s_E$ 
be the wave operator, the scattering operator, and the scattering symplectomorphism corresponding to 
the pair of Hamiltonians $H_0$, $H_s$. Denote
$$
F_s(x)=\frac{d}{ds}H_s(x), \quad
Y_s(x)=\frac{d}{ds}X_s(x);
$$ 
then we have 
$i(Y_s)\omega=-dF_s$.

2. 
Define 
\begin{equation}
\label{i16}
a_s(x)=\int_{-\infty}^\infty F_s\circ \Phi_t^s(x)dt, 
\quad x\in A_s.
\end{equation}
By our assumptions, $F_s=0$ on $A_s$ outside a compact set. 
Thus, the integration in \eqref{i16} is performed over a bounded set of $t$. 
It follows that $a_s\in C^\infty(A_s)$.

It is clear that $a_s\circ \Phi_t^s=a_s$ for all $t\in\R$, and therefore 
one can define $\wt a_s\in C^\infty(\wt A_s)$ such that 
$\wt a_s\circ \pi_s=a_s$. Below we will prove that 
\begin{equation}
i\left(\left(\frac{d}{ds}\wt S^s_E\right)\circ (\wt S_E^s)^{-1}\right)\wt \omega_0=-d(\wt a_s\circ \wt W_+^s)
\quad
\text{on $\wt A_0$,}
\label{i17}
\end{equation}
i.e. the family of Hamiltonians $\wt a_s\circ \wt W_+^s$, $s\in[0,1]$, generates the 
symplectomorphism $\wt S^1_E$. 

Note that  the family $\wt a_s\circ \wt W_+^s$ is compactly supported. 
Indeed, let $K\subset \N$ be a compact set such that 
$$
 A_0(E)\cap \bigl(\cup_{s\in[0,1]} \supp(H_s-H_0)\bigr)\subset K
$$
and let $\wt K=\pi_0(K)$. Then 
it is easy to see that for $x\in \wt A_0\setminus \wt K$, one has
$\wt a_s\circ \wt W_+^s(x)=0$. 
Thus, in order to prove the theorem, it suffices to check \eqref{i17}.

3. In order to prove  \eqref{i17}, below we will check 
the identity
\begin{equation}
i\left(\left(\frac{d}{ds}S^s_E\right)\circ (S_E^s)^{-1}\right)\omega=-d(a_s\circ W_+^s)
\quad \text{on $A_0$.}
\label{i18}
\end{equation}
The identity \eqref{i18} can be written as 
\begin{equation}
\pi_0^*\left( 
i(\frac{d}{ds}\wt S^s_E)\wt \omega_0+d(\wt a_s\circ \wt W_+^s)
\right) =0.
\label{i19}
\end{equation}
Since $\pi_0$ is surjective, \eqref{i19} yields \eqref{i17}. Thus, it suffices to prove \eqref{i18}. 
But first we need to prepare two auxiliary identities: \eqref{i20} and \eqref{i21}.

4. By a direct differentiation, we have
$$
\frac{d}{ds}\frac{d}{dt}\Phi_{-t}^0\circ\Phi_t^s(x)|_{s=0}
=
d_{\Phi_t^0(x)}\Phi^0_{-t}(Y_0\circ \Phi_t^0(x)).
$$
Interchanging the order of differentiation in the l.h.s. and integrating over $t$ 
from $0$ to $T$, we obtain
$$
\frac{d}{ds}\Phi_{-T}^0\circ\Phi_T^s(x)|_{s=0}
=
\int_0^T
d_{\Phi_t^0(x)}\Phi^0_{-t}(Y_0\circ \Phi_t^0(x))dt.
$$
It follows that 
$$
\frac{d}{ds}\Phi_T^s(x)|_{s=0}
=
\int_0^T
d_{\Phi_t^0(x)}\Phi^0_{T-t}(Y_0\circ \Phi_t^0(x))dt.
$$
In the same way, for any $s\in[0,1]$, we get
\begin{equation}
\frac{d}{ds}\Phi_T^s(x)
=
\int_0^T  d_{\Phi_t^s(x)}\Phi_{T-t}^s (Y_s\circ \Phi_t^s(x)) \, dt.
\label{i20}
\end{equation}

5. Let us define a vector field $V_s$ on $A_s$ by 
$$
V_s(x)=\int_{-\infty}^\infty d_{\Phi_t^s(x)}\Phi^s_{-t} (Y_s\circ \Phi_t^s(x))dt.
$$
Using the identities $i(Y_s)\omega=-dF_s$ and 
$$
\omega(d\Phi^s_{-t}\xi,\eta)=\omega(\xi,d\Phi_t^s\eta),
$$
one easily verifies that 
\begin{equation}
i(V_s)\omega=-da_s
\quad \text{on}\quad A_s
\label{i21}
\end{equation}
for all $s\in[0,1]$. 

6. Using formulas \eqref{a7} for $W_+^s$, \eqref{a1} for $S_E^s$ and \eqref{i20},
we obtain
\begin{multline*}
\left(\frac{d}{ds}S_E^s\right) \circ (S_E^s)^{-1}
=
d\Phi^0_{-T} \left( \left(\frac{d}{ds}\Phi_{2T}^s \right) \circ \Phi^s_{-2T}\circ \Phi_T^0\right)
=
d\Phi^0_{-T} \left( \left(\frac{d}{ds}\Phi_{2T}^s\right) \circ \Phi^s_{-T}\circ W_+^s \right)
\\
=
d((W_+^s)^{-1})\int_0^{2T} d\Phi_{T-t}^s(Y_s\circ \Phi_{t-T}^s\circ W_+^s)dt
=
(d W_+^s)^{-1}(V_s\circ W_+^s).
\end{multline*}
Using \eqref{i21}, we obtain
$$
i\left( \left(\frac{d}{ds}S_E^s\right)\circ (S_E^s)^{-1}\right)\omega
=
i( (dW_+^s)^{-1}(V_s\circ W_+^s)\omega=-d(a_s\circ W_+^s),
$$
as required in \eqref{i18}.

\section*{Appendix A: Key formulas in quantum scattering}
\label{sec.appa}
\renewcommand{\theequation}{A.\arabic{equation}}
\renewcommand{\thetheorem}{A.\arabic{theorem}}
\renewcommand{\thesubsection}{A.\arabic{subsection}}
\setcounter{theorem}{0}
\setcounter{equation}{0}

Here we collect those definitions and formulas in quantum scattering theory which are relevant
to the rest of the paper. We do not make any attempt at being rigourous or even precise
about the required assumptions. Our cavalier approach will probably horrify experts
in quantum scattering but this collection of formulas might be useful 
for the purposes of comparison of ``classical'' and ``quantum'' cases.

Let $H_0$ and $H$ be self-adjoint operators in a Hilbert space $\mathcal H$. 
In order to simplify our discussion, let us assume that both $H_0$ and $H$ 
have purely absolutely continuous spectrum (see e.g. \cite[Section~VII.2]{RS1}).
The wave operators $W_\pm:\mathcal H\to\mathcal H$ are defined by
$$
W_\pm \psi=\lim_{t\to\pm\infty} e^{itH}e^{-itH_0}\psi,
\quad \psi\in\mathcal H,
$$
whenever these limits exist. The wave operators are easily seen to be isometric and intertwine
$H$ and $H_0$: $W_\pm H_0=HW_\pm$. 
Completeness of the wave operators is the relation
$$
\Ran W_+=\Ran W_-=\mathcal H.
$$
If $H$ has a non-empty pure point or singular continuous spectrum, then $\mathcal H$ in the 
right hand side of the last relation has to be replaced by the absolutely continuous subspace of $H$. 
See \cite[Sections~2.1, 2.3]{Yafaev} for the details. 

If the wave operators exist and are complete, one defines the scattering operator
$S=W_+^{-1}W_-$. 
The scattering operator is unitary and commutes with $H_0$: $SH_0=H_0S$. 
It follows that $S$ is diagonalised by the direct integral of the spectral decomposition of $H_0$:
\begin{equation}
\mathcal H=\int^\oplus_{\Spec(H_0)} {\mathfrak h} (E)dE,
\quad
(H_0f)(E)=Ef(E),
\quad
(Sf)(E)=S_Ef(E).
\label{di}
\end{equation}
Here $S_E$ is a unitary operator in the fibre space $\mathfrak h(E)$; $S_E$ is called 
the scattering matrix. See \cite[Section~2.4]{Yafaev} for the details. 

Let $P^{(k)}$ be a family of operators in $\mathcal H$ such that $P^{(k)}\psi\to\psi$ 
as $k\to\infty$ for any $\psi\in\mathcal H$. Let $\mathcal T^{(k)}$ be the operator defined by 
$$
(\mathcal T^{(k)}\psi,\psi)
=
\int_{-\infty}^\infty \norm{P^{(k)}e^{-itH}W_-\psi}^2\  dt
-
\int_{-\infty}^\infty \norm{P^{(k)}e^{-itH_0}\psi}^2\  dt,
$$
assuming that these integrals exist. 
Then $\mathcal T^{(k)}$ commutes with $H_0$ and therefore is diagonal
with respect to the direct integral decomposition \eqref{di}.
Let $\mathcal T^{(k)}_E:\mathfrak h(E)\to \mathfrak h(E)$ be  the corresponding 
fibre operator. Then the limit $\mathcal T_E=\lim_{k\to\infty} \mathcal T^{(k)}_E$,
whenever it exists, is called the global time delay operator, and 
$T_E=\Tr \mathcal T_E$ is called the global time delay. 
See \cite{Robert1} for the details. 

The spectral shift function $\xi(E)$ is defined by the relation
$$
\xi(E)=\Tr(\chi_{(-\infty, E)}(H_0)-\chi_{(-\infty,E)}(H))
$$
which has to be understood in a certain regularised sense; see \cite[Section~8.2]{Yafaev}
for the details.
The existence of the spectral shift function and the global time delay requires some
trace class assumptions on $H$ and $H_0$, such as $H-H_0\in\text{Trace class}$.

The Birman-Krein formula reads
\begin{equation}
\det S_E=e^{-2\pi i\xi(E)};
\label{BK}
\end{equation}
see \cite[Section~8.4]{Yafaev} for the details. 
The Eisenbud-Wigner formula reads
\begin{equation}
\frac{d}{dE}\Im\log\det S_E=T_E;
\label{EW}
\end{equation}
see \cite{Robert1} for the details.

\section*{Appendix B: Symplectic diffeomorphisms and the Calabi invariant}
\label{sec.app}
\renewcommand{\theequation}{B.\arabic{equation}}
\renewcommand{\thetheorem}{B.\arabic{theorem}}
\renewcommand{\thesubsection}{B.\arabic{subsection}}
\setcounter{theorem}{0}
\setcounter{equation}{0}

\subsection{Symplectomorphisms}  
We recall some notation and preliminaries from symplectic topology; see e.g. \cite{McDuff} 
for the details. 

Let $(\M,\omega)$ be a non-compact $2m$-dimensional symplectic manifold, 
possibly with boundary.
We need the following notation: 

$\Symp( \M)$ is the group of all symplectomorphisms (=symplectic diffeomorphisms) on $ \M$.

$\Symp^c( \M)$ is the group of all symplectomorphisms of $ \M$ with compact support, 
$\supp\psi=\Clos \{x\mid \psi(x)\not=x\}$. 

$\Symp^c_0( \M)$ is the path connected component of the identity map in $\Symp^c( \M)$. 

$\Ham^c( \M)$ is the set of all compactly supported Hamiltonian symplectomorphisms of $ \M$. 
This means that $\psi\in\Ham^c( \M)$ can be constructed as a time one flow of a family of time-dependent
compactly supported Hamiltonians. More precisely, $\psi\in\Ham^c( \M)$ means that there exists a smooth family $h_t$, 
$t\in[0,1]$ of Hamiltonians on $\M$ such that $\cup_{t\in[0,1]}\supp h_t$ lies in a compact set and if 
$\psi_t$ is the corresponding flow and $X_t$ the corresponding vector field, 
\begin{equation}
\frac{d}{dt}\psi_t=X_t\circ\psi_t, 
\quad
i(X_t) \omega=-dh_t,
\quad
\psi_0=id,
\label{e3}
\end{equation}
then $\psi=\psi_1$. 

It is easy to see  (cf. \cite[Section~10]{McDuff}) that
\begin{equation}
\Ham^c( \M)\subset\Symp_0^c( \M)\subset \Symp^c( \M). 
\label{e8}
\end{equation}

\subsection{The Calabi invariant}
Let us assume that $\M$ is  exact, i.e. there is a 1-form $\alpha$ such that $ \omega=d\alpha$. 
We recall the definition of the Calabi invariant $\CAL$; for the details, see \cite{McDuff}. 
For our purposes we need to define $\CAL$ on a wider set of symplectomorphisms
than it is usually done. Let $\Dom(\CAL,  \M)$ be the set of all $\psi\in\Symp^c( \M)$ 
such that there exists a 1-form $\alpha$ and $f\in C_0^\infty( \M)$ with $d\alpha= \omega$ 
and $\alpha-\psi^*\alpha=df$. 
In this case we will say that $f$ is a generating function of $\psi$. 

If $\psi\in\Dom(\CAL,  \M)$ and $f$ is a generating function of $\psi$, let us 
define
\begin{equation}
\CAL(\psi)=\frac{1}{m+1}\int_{\M}f(x) \frac{\omega^m(x)}{m!}. 
\label{e2}
\end{equation}
Note that our sign conventions and normalisation differ from those of \cite{McDuff}. 

A symplectomorphism can have many generating functions. However, we have
\begin{proposition}
(i) 
$\CAL(\psi)$ is independent of the choice of a generating function $f$ of $\psi\in \Dom(\CAL,\M)$.

(ii)
If $\alpha-\psi^*\alpha=df$ with $f\in C_0^\infty(\M)$, then $\CAL(\psi)$ can  be calculated as
\begin{equation}
\CAL(\psi)=\frac1{(m+1)!}\int_\M (\psi^* \alpha)\wedge \alpha \wedge \omega^{m-1}.
\label{e5}
\end{equation}

(iii)
Suppose $\psi\in\Ham^c( \M)$ is generated by a family of Hamiltonians $\{h_t\}$, see \eqref{e3}. 
Then $\psi\in \Dom(\CAL,\M)$ and 
$$
\CAL(\psi)=\int_0^1 dt\int_{ \M} h_t(x) \frac{\omega^m(x)}{m!}.
$$
\end{proposition}

\begin{proof}
(i) This is a well known calculation, see e.g. \cite[Lemma 10.27]{McDuff}. 
Suppose that $\alpha_j-\psi^*\alpha_j=df_j$, $f_j\in C_0^\infty( \M)$, 
$d\alpha_j= \omega$, $j=1,2$. 
Denote $\beta=\alpha_2-\alpha_1$, $g=f_2-f_1$; then we have $d\beta=0$, 
$\beta-\psi^*\beta=dg$.  We need to prove that $\int_{ \M} g  \omega^m=0$. 

As in \cite[Lemma 10.27]{McDuff}, we have
\begin{multline*}
\int_{ \M}g \omega^m
=
\int_{\M}g \ d\alpha_1\wedge  \omega^{m-1}
=
\int_{ \M}(d(g\alpha_1)-(dg)\wedge \alpha_1)\wedge  \omega^{m-1}
=
-\int_{ \M}(dg)\wedge \alpha_1\wedge  \omega^{m-1}
\\
=
\int_{ \M}(\psi^*\beta-\beta)\wedge \alpha_1\wedge  \omega^{m-1}
=
\int_{ \M}(\psi^*\beta)\wedge \alpha_1\wedge  \omega^{m-1}
-
\int_{ \M}\psi^*(\beta\wedge \alpha_1\wedge  \omega^{m-1})
\\
=
\int_{ \M}\psi^*\beta\wedge (\alpha_1-\psi^*\alpha_1)\wedge  \omega^{m-1}
=
\int_{ \M}\psi^*\beta\wedge df_1 \wedge  \omega^{m-1}
=
-\int_{ \M}d(f_1\psi^*\beta\wedge  \omega^{m-1})=0,
\end{multline*}
as required. 

(ii) Similarly to the previous calculation, using Stokes formula, we have
\begin{multline*}
(m+1)! \CAL(\psi)
=
\int_\M f (d\alpha)\wedge \omega^{m-1}
=
\int_\M f d(\alpha \wedge \omega^{m-1})
=
-\int_\M df\wedge \alpha\wedge \omega^{m-1}
\\
=
\int_\M(\psi^*\alpha-\alpha)\wedge \alpha\wedge \omega^{m-1}
=
\int_\M \psi^*\alpha\wedge \alpha\wedge \omega^{m-1},
\end{multline*}
since $\alpha\wedge\alpha=0$.

(iii) Is well known; see \cite[Lemma~10.27]{McDuff}. 
\end{proof}

\begin{remark}
The Calabi invariant is usually defined as a map $\Ham^c( \M)\to \R$, in which case it is 
a homomorphism; see \cite{McDuff}. On the domain $\Dom(\CAL,  \M)$
this is in general not the case: it is not difficult to show that 
$\Dom(\CAL,  \M)$ is, in general, not a subgroup of $\Symp^c( \M)$. 
Example~\ref{exa.4}   below shows that in general neither of the two sets $\Symp^c_0( \M)$, $\Dom(\CAL,\M)$ 
is a subset of the other one.
\end{remark}

\begin{example}\label{exa.4}
Let $( \M,\omega)$ be $T^*\mathbb S^1$ with the canonical symplectic structure of the cotangent
bundle (see e.g. \cite[Section~3.1]{McDuff}). Let $(s,\theta)\in\R\times[0,2\pi]$ be the coordinates
in $T^*\mathbb S^1$. All possible one-forms $\alpha$ such that $d\alpha=\omega$ 
can be described as $\alpha=sd\theta+\gamma d\theta+dg$, where $\gamma\in\R$ and 
$g\in C^{\infty}(T^*\mathbb S^1)$. 
Consider the symplectic diffeomorphism  $\psi$ of 
$T^*\mathbb S^1$, defined by 
$$
\psi: (s,\theta)\mapsto (s,\theta+\phi(s)),
$$
where $\phi\in C^\infty(\R)$. Consider the following cases: 

(i) 
Suppose $\phi(s)=0$ for $s\leq-1$ and $\phi(s)=2\pi$ for $s\geq1$, and $\int_{\R}s\phi'(s)ds=0$.
Then $\psi$ is not homotopic to identity, but  
$\psi^*(sd\theta)-sd\theta=s\phi'(s)ds$ and so $\psi\in\Dom(\CAL,T^*\mathbb S^1)$. 
Thus, $\Dom(\CAL,T^*\mathbb S^1)$ is not a subset of $\Symp^c_0(T^*\mathbb S^1)$.

(ii)
Suppose $\phi\in C_0^\infty(\R)$, and  
$\int_{\R}s\phi'(s)ds\not =0$.
Then it is easy to see that $\psi\in\Symp^c_0( T^*\mathbb S^1)$ but  $\psi\not\in\Dom(\CAL,T^*\mathbb S^1)$. 
\end{example}

\subsection{$\log\det$ and $\CAL$}\label{sec.app3}
Here, without any attempt at being rigourous, we point out an analogy 
between the Calabi invariant of a symplectic map and the logarithm of the determinant
of a unitary operator.  This analogy helps to understand the relation between 
Theorems~\ref{maintheorem}, \ref{th2}  and their ``quantum" counterparts.

In order to make our discussion concrete, 
suppose $\M=\R^{2n}$ and let us use the Weyl quantisation procedure. 
That is, for a real valued function $h\in C_0^\infty(\R^{2n})$, let us define the self-adjoint operator
$$
(Hu)(q)
=
\int_{\R^{2n}} 
e^{i \langle q-q',p\rangle }h(\tfrac{q+q'}2,p)u(q')d^n q' d^n p.
$$
Then, clearly, 
\begin{equation}
\Tr H=\int_{\R^{2n}} h(q,p)\  d^nq\  d^np=\int_{\R^{2n}} h(x)\frac{\omega^n(x)}{n!}.
\label{f1}
\end{equation}
This connection between trace and phase space integral lies at the heart of the correspondence
between quantum and classical mechanics.

Further, 
let $U$ be the unitary operator obtained from $H$ by means of exponentiation: $U=\exp(-iH)$. 
$U$ can be regarded as a time one map corresponding to the differential equation $i\frac{d}{dt}U(t)=HU(t)$. 
The analogue of this procedure is taking the time one map of the Hamiltonian flow $\Phi_t$ generated
by $h$. Then we have
$$
-\Im \log\det U=\Tr H = \int_{\R^{2n}} h(x) \frac{\omega(x)^n}{n!}=\CAL(\Phi_1).
$$
In other words, we have a diagram
$$
\begin{CD}
h @>Quantisation>> H
\\
@V
\genfrac{}{}{0pt}{1}{classical}{dynamics}  VV @VV\genfrac{}{}{0pt}{1}{quantum}{dynamics}  V
\\
\Phi_1  @.   e^{iH}
\\
@VVV @VVV
\\
\CAL(\Phi_1) @= - \Im \log\det e^{iH}
\end{CD}
$$
This to some extent explains the analogy between $-\Im \log\det$ and $\CAL$.

\section*{Acknowledgements}
Research was partially supported by the London Mathematical Society.
A.P. is grateful to  H.~Dullin,
A.~Gorodetski, M.~Hitrik, and  A.~Strohmaier for useful discussions and references 
to the literature and to N.~Filonov for reading the manuscript and making a number
of very helpful remarks.

\end{document}